\documentclass[a4paper]{article}
\usepackage{amsmath,amsthm,amssymb,enumerate,hyperref,xcolor}
\usepackage[ruled,linesnumbered]{algorithm2e}
\usepackage[legacycolonsymbols]{mathtools}
\usepackage{booktabs}
\providecommand\given{}
\newcommand\SetSymbol[1][]{%
  \nonscript\:#1\vert
  \allowbreak
  \nonscript\:
  \mathopen{}}
\DeclarePairedDelimiterX\Set[1]{\{}{\}}{%
  \renewcommand\given{\SetSymbol[\delimsize]}
  #1}

\DeclarePairedDelimiterXPP\pospart[1]{}{(}{)}{^+}{#1}
\DeclarePairedDelimiterXPP\negpart[1]{}{(}{)}{^-}{#1}

\newcommand\R{\mathbb{R}}

\newcommand\N{\mathbb{N}}
\newcommand\PP{\mathcal{P}}
\newcommand\QQ{\mathcal{Q}}

\newtheorem{thm}{Theorem}[section]
\newtheorem{prop}[thm]{Proposition}
\newtheorem{cor}[thm]{Corollary}
\theoremstyle{definition}

\newtheorem{ex}[thm]{Example}
\newtheorem{rem}[thm]{Remark}

\DeclareMathOperator*{\vertices}{vert}
\DeclareMathOperator*{\Int}{int}

\DeclareMathOperator*{\gr}{gr}

\DeclareMathOperator*{\conv}{conv}
\DeclareMathOperator*{\cone}{cone}
\DeclareMathOperator*{\dom}{dom}

\DeclareMathOperator*{\aff}{aff}
\DeclareMathOperator*{\std}{^\text{\rm std}}
\DeclareMathOperator*{\homog}{^\text{\rm hom}}

\makeatletter
\newcommand{\leqnomode}{\tagsleft@true\let\veqno\@@leqno}
\newcommand{\reqnomode}{\tagsleft@false\let\veqno\@@eqno}
\makeatother

\title{A solution method for arbitrary polyhedral convex set optimization problems}
\author{Andreas L{\"o}hne\thanks{Friedrich Schiller University Jena, Faculty of Mathematics and Computer Science, 07737 Jena, Germany, andreas.loehne@uni-jena.de}}

\begin{document}
\maketitle

\begin{abstract} \noindent We provide a solution method for the polyhedral convex set optimization problem, that is, the problem to minimize a set-valued mapping $F:\R^n \rightrightarrows \R^q$ with polyhedral convex graph with respect to a set ordering relation which is generated by a polyhedral convex cone $C \subseteq \R^q$. The method is proven to be correct and finite without any further assumption to the problem.

\medskip

\end{abstract}

\section{Introduction}

Polyhedral convexity is closely related to linear programming. Minimizing an extended real-valued polyhedral convex function of $n$ real variables over $\R^n$ can be expressed equivalently by a linear program and vice versa, for some details see, e.g., \cite[Section 3]{CirLoeWei19}. George Dantzig ends the abstract of his 1982 paper \cite{Dantzig82} about the origin of linear programming with the statement: 
\begin{quotation}
	{\em ``Linear Programming is viewed as a revolutionary development giving us the ability for the first time to state general objectives and to find, by means of the simplex method, optimal policy decisions for a broad class of practical decision problems of great complexity.''}
\end{quotation}

\noindent
One year later, the abstract of \cite{Dantzig83} starts in a very similar way but a programmatic note was added:

\begin{quotation}
{\em ``In the real world, planning tends to be ad hoc because of the many special-interest
groups with their multiple objectives. Much work remains to develop a more disciplined
infrastructure for decision making in which the full potential of mathematical programming models could be realized.''}
\end{quotation}

{\em Multi-objective optimization} and {\em vector linear programming} can be seen as some important part of such a development, see, e.g., \cite{Zeleny74, Ehrgott05, Loe11}. Again there is a strong connection to polyhedral convexity, as these classes are shown to be equivalent to the {\em polyhedral projection problem} \cite{LoeWei16}.  
{\em Polyhedral convex set optimization} is another element in this series of problem classes as it generalizes vector linear programming. 

Set optimization in the sense of optimizing a set-valued objective mapping on the basis of a set ordering, i.e., an ordering on the power set of, e.g., $\R^q$, started in the late 1990s by Daishi Kuroiwa, e.g., in \cite{Kuroiwa98}, see \cite[Section 2.2]{HamHeyLoeRudSch15} for further references. Surprisingly, this development began on a rather theoretical level with a very general problem setting. So it took several years until the first connections to applications were established, where the specific case of polyhedral convex problems played an increasing role, see, e.g., \cite[Section 7.5]{HamHeyLoeRudSch15} for a choice of references. We believe that one key to new more extensive decision making tools lies in a thorough development of this simplest class of set optimization problems, the {\em polyhedral convex set optimization problems}, which in some sense can be seen as {\em ``set linear programs''}.  
 
A set-valued mapping $F:\R^n\rightrightarrows \R^q$ is said to be \emph{polyhedral convex} if its \emph{graph}
$$ \gr F \coloneqq \Set{(x,y) \in \R^n\times \R^q \given y \in F(x)}$$
is a convex polyhedron. 
Thus, every polyhedral convex set-valued mapping can be represented by two integers $n,q \in \N$, which  specify the dimensions of the domain and codomain spaces, together with a convex polyhedron $\gr F \subseteq \R^n \times \R^q$, as follows:
$$ F: \R^n \rightrightarrows \R^q, \quad F(x) = \Set{y \in \R^q \given (x,y) \in \gr F}.$$

The problem to minimize $F$ with respect to the partial ordering $\supseteq$ is called a \emph{polyhedral convex set optimization problem}. Note that the inverse set inclusion $\supseteq$ is understood as ``less than or equal to''.
The objective mapping $F$ is sometimes given together with a nonempty polyhedral convex cone $C$, called the \emph{ordering cone}.
This leads to the problem to minimize the set-valued mapping 
$$F_C:\R^n\rightrightarrows \R^q,\quad F_C(x) \coloneqq F(x) + C$$
with respect to $\supseteq$.
This problem is stated as
\leqnomode
\begin{gather}\label{eq:p}\tag{P}
  \min F(x) + C\quad \text{ s.t. }\; x \in \R^n.
\end{gather}
\reqnomode
If $C\subseteq \R^q$ is a nonempty polyhedral convex cone, then 
$$ y^1 \leq_C y^2 \; \iff \; y^2 \in \Set{y^1} +  C$$
defines a reflexive and transitive ordering on $\R^q$, which can be lifted to a reflexive and transitive ordering on the power set of $\R^q$ defined by
\begin{equation}\label{eq:setrel}
	Y^1 \preccurlyeq_C Y^2  \;\iff\; Y^2 \subseteq Y^1 + C\text{.}
\end{equation}
Problem \eqref{eq:p} can be seen as minimizing $F$ with respect to $\preccurlyeq_C$.
Note that the original problem of minimizing $F$ without mentioning any ordering cone is re-obtained by the choice $C=\Set{0}$.

Problem \eqref{eq:p} is called {\em bounded} if there exists a bounded set $L \subseteq \R^q$ (without loss of generality, $L$ can be a polytope) such that
$$ \forall x \in \R^n:\quad L + C \supseteq F(x).$$
The \emph{upper image} of \eqref{eq:p} is the convex polyhedron defined by
$$ \PP \coloneqq C + \bigcup_{x \in \R^n} F(x),$$ 
see, e.g., \cite{HeyLoe}. It can be seen as the infimum of $F$ over $\R^n$ with respect to $\preccurlyeq_C$, see, e.g., \cite{HamHeyLoeRudSch15}.

A solution method for bounded problems has been developed in \cite{LoeSch13}. A {\em solution} of a bounded instance of problem \eqref{eq:p} is a finite set of vectors 
$$ \bar S \coloneqq \Set{x^1,\dots,x^k} \subseteq \dom F \coloneqq \Set{x \in \R^n \given F(x) \neq \emptyset} $$ 
such that
\begin{equation}\label{eq:infimizer}
	\PP = C + \conv \bigcup_{\bar x \in \bar S} F(\bar x)  
\end{equation}
holds and for all $\bar x \in \bar S$ one has
\begin{equation}\label{eq:minimizer}
	\not\exists x\in \R^n:\; F(x) + C \supsetneq F(\bar x) + C.
\end{equation}
A set $\bar S \subseteq \R^n$ satisfying \eqref{eq:infimizer} is called an {\em infimizer} and an element $\bar x \in \dom F$ satisfying \eqref{eq:minimizer} is called a {\em minimizer} of the  set optimization problem \eqref{eq:p} in case it is bounded. 

In the first part of this paper, Sections \ref{sec:minimizer} and \ref{sec:alg_bounded}, we give a simplified variant of the algorithm for bounded problems from \cite{LoeSch13} and we prove correctness and finiteness of this variant.
In the second part, Sections \ref{sec:alg_unbounded}, \ref{sec:std} and \ref{sec:empty}, we extend the algorithm and the related results by dropping all assumptions of the problem, except polyhedral convexity of $F$ and $C$. 

The extension to not necessarily bounded problems is based on some recent results on the existence of solutions of polyhedral convex set optimization problems \cite{Loe23exsol} as well as the concept of the {\em natural ordering cone} of a polyhedral convex set optimization problem \cite{Loe23natcon}, which is recalled in Section \ref{sec:alg_unbounded}. 

The solution concepts we use in this paper are based on the two components {\em infimum attainment} and {\em minimality}. This approach is due to \cite{HeyLoe11} with contributions by Andreas H. Hamel. 
Polyhedral variants of the solution concept have been introduced in \cite{Loe11} for vector optimization problems and in \cite{LoeSch13} for bounded polyhedral convex set optimization problems. 
The solution concept for not necessarily bounded polyhedral convex set optimization problems we use here, see Section \ref{sec:alg_unbounded}, appeared in \cite{Loe23exsol} and was contributed by Andreas H. Hamel and Frank Heyde.

There are two main approaches in the literature (for references see, e.g., \cite{HamHeyLoeRudSch15,KhaTamZal15}) for optimization problems with set-valued objective mappings and an ordering given by a set relation (like the one in \eqref{eq:setrel}):
first the {\em set approach}, which aims to compute all minimizers or some minimizers without any requirement of infimum attainment, and secondly, the {\em complete lattice approach}, where the goal is to compute a set of minimizers which provide attainment of the infimum.
This paper mainly focuses on the complete lattice approach, but the method we present below in Section \ref{sec:minimizer} can also be used to compute minimizers without the requirement of some kind of infimum attainment.
It computes, for an arbitrary point $\bar y \in \PP$ some minimizer $\bar x$ of \eqref{eq:p} with $\bar y \in F(\bar x)+C$. Note that, as shown in \cite{Loe23exsol}, this task can also be realized by solving one single linear program provided an inequality representation of the graph of $F_C$ is known. This is, however, usually not the case and the computation of such a representation is not tractable for problems with many variables. 

To the best of our knowledge, so far there is no solution method known in the literature for unbounded polyhedral convex set optimization problems. Let us discuss some related references.
In \cite{Jahn15}, minimizers are obtained as minimizers of an associated vector optimization problem, however, one with infinitely many (or in the polyhedral case at least with many) objectives. The recent papers \cite{EicQuiRoc22,EicRoc23} provide methods to compute all {\em weakly minimal solutions} (which is in our terminology some kind of {\em weak minimizer}) for set optimization problems by solving an associated vector optimization problem. The set optimization problems in \cite{EicQuiRoc22,EicRoc23} are in many aspects more general than those we consider here. Moreover, different set ordering relations are considered, whereas we consider only one of them. The results can be extended to those other set relations which can be expressed as an inclusion of two sets. In \cite[Theorem 3.3]{EicRoc23} weakly minimal solutions (weak minimizers) of a set optimization problem are characterized by weakly minimal solutions (weak minimizers) of an associated infinite dimensional vector optimization problem. Discretization together with an error analysis yield a method to compute weakly minimal solutions (weak minimizers) by computing weakly minimal solutions of a vector optimization problem with finitely many objectives. In order to point out the impact of this paper, let us boil down the objective of \cite{EicRoc23} to our special setting. Motivated by the characterization in \cite[Proposition 2.10]{EicRoc23} we define a {\em weak minimizer} $\bar x \in \R^n$ of \eqref{eq:p} by the condition
\begin{equation*}
	\not\exists x\in \R^n:\; F(x) + \Int C \supseteq F(\bar x).
\end{equation*}
Below in Section \ref{sec:alg_bounded} we compute for each vertex $\bar y$ of the upper image $\PP$ some minimizer $\bar x$ of \eqref{eq:p} such that $\bar y \in F(\bar x)+ C$. A weak minimizer $\bar x$ of \eqref{eq:p} with $\bar y \in F(\bar x)+C$ is easier to obtain. Since $F(x)+\Int C$ is an open subset of $\PP$, a point $\bar x \in \R^n$ is already a weak minimizer of \eqref{eq:p} if $F(\bar x)$ contains a boundary point of $\PP$. Thus for a vertex $\bar y$ of $\PP$, every $\bar x \in \R^n$ with $\bar y \in F(\bar x)+C$ is a weak minimizer. An example where for a vertex $\bar y$ of $\PP$ some $\bar x$ with $\bar y \in F(\bar x)+C$ is not a minimizer can be found in \cite{LoeSch13}. 

Bibliographic notes on set optimization and its applications can be found, e.g., in \cite{HamHeyLoeRudSch15,KhaTamZal15}. Some further references on solution methods for set optimization problems are provided in \cite{EicRoc23}.

\section{A motivating example}\label{sec:mot}

In this section, we describe a very specific setting that involves set-valued risk measures, as introduced in \cite{JouMebTou04} and further extended in \cite{HamHey07, HamHeyRud11}. The concept of \emph{risk} provides its full potential only in a setting with random variables. However, to maintain simplicity in this example, we replace them by deterministic variables.

Consider \(n\) assets such as stocks, currencies, and commodities. A portfolio of these \(n\) assets can be described by a vector \(x \in \mathbb{R}^n\). A component \(x_i\) represents the units of asset \(i\) in the portfolio. A portfolio is considered to be \emph{solvent} if it can be traded into a nonnegative portfolio under the current market conditions. The set of all solvent portfolios is called the \emph{solvency cone}. It describes the current market conditions for trading the \(n\) assets.

The first \(q\) of the \(n\) assets are considered to be \emph{eligible} for risk compensation. The risk of the portfolio \(x \in \mathbb{R}^n\) in terms of the \(q \leq n\) eligible assets can be seen as the set of all portfolios \(y \in \mathbb{R}^q\) of eligible assets such that
\[
x + \begin{pmatrix} y \\ 0 \end{pmatrix} \in K_2,
\]
where the solvency cone \(K_2 \subseteq \mathbb{R}^n\) describes the market conditions for risk compensation.

We consider a bank that wants to sell a portfolio \(x \in \mathbb{R}^n\) to various customers. The bank has an initial portfolio \(\bar{x}\) which can be traded into the portfolio \(x\) to be sold. For this trade, certain market conditions hold, which are described by a solvency cone \(K_1 \subseteq \mathbb{R}^n\). The bank wants to obtain $x$ from $\bar x$ by trading on the market without adding any capital (but removing capital is possible). This means that any \(x \in \Set{\bar{x}} - K_1\) can be chosen by the bank.

Two decision-making processes are involved. The bank has to choose \(x \in \Set{\bar{x}} - K_1\). The second decision has to be made by the customers. Each customer has to compensate the risk by selecting a portfolio \(y \in \mathbb{R}^q\), i.e., a customer has to choose an element from the set
\[
R(x) \coloneqq \left\{ y \in \mathbb{R}^q \mid x + \begin{pmatrix} y \\ 0 \end{pmatrix} \in K_2 \right\}.
\]
The bank aims to satisfy its customers by maximizing the set \(R(x)\) with respect to set inclusion. This ensures that different customers can compensate the risk in their own way. It provides maximal flexibility to the customers. So the bank has a \emph{preference for flexibility} \cite{Kreps79,HamLoe20}.

It is sufficient to enlarge the set \(R(x) + \mathbb{R}^q_+\) instead of \(R(x)\) because risk compensation by component-wise smaller portfolios of eligible assets is usually preferred. Adding the cone $\R^q_+$ directs attention towards enlargements with respect to these preferred options for risk compensation.

This setting leads to the following set optimization problem:
\[
\text{Maximize } R(x) + \mathbb{R}^q_+ \quad \text{s.t.} \quad x \in \bar{x} - K_1
\]
with respect to set inclusion \(\subseteq\), or equivalently,
\[
\text{Minimize } R(x) + \mathbb{R}^q_+ \quad \text{s.t.} \quad x \in \bar{x} - K_1
\]
with respect to the superset inclusion \(\supseteq\).

For polyhedral convex solvency cones, we obtain a polyhedral convex set optimization problem. In the numerical examples later in Section \ref{sec:num} we see that this setting can even lead to an \emph{unbounded} set optimization problem. This motivates extending the solution methods to this case.

Another field of application is multi-stage stochastic linear programming with multiple objectives as recently discussed in \cite{HamLoe24}.

\section{Computing minimizers}\label{sec:minimizer}

In this section we present the core of all algorithms of this paper. 
If $\bar y$ denotes some point in the upper image $\PP$ of the set optimization problem \eqref{eq:p},
we compute a minimizer $\bar x$ of \eqref{eq:p} such that $\bar y \in F(\bar x)+C$. 

For a point $\bar x \in \dom F$ and a vector $w \in C^*$,
where $C^*\coloneqq \Set{z \in \R^q \given \forall y \in C: z^T y \leq 0}$ denotes the polar cone of $C$,
we consider the linear program
\leqnomode
\begin{gather}\label{eq:lp}\tag{LP$(F,w,\bar x)$}
	 \max_{x,y} w^T y \quad \text{ s.t. } \quad  \left\{ \begin{array}{rcl}
		y & \in & F(x)+C \\
		F(\bar x)+C & \subseteq & F(x)+C\text{.}
	\end{array}\right.
\end{gather}
\reqnomode
Problem \eqref{eq:lp} is feasible for all $\bar x \in \dom F$.
The inclusion in the second constraint can be resolved by computing a {\em V-representation} of the polyhedral convex set $F(\bar x)+ C$, i.e., a representation as the sum of the convex hull of finitely many points and the conic hull of finitely many directions, see, e.g., \cite{CirLoeWei19}.
Since all values of $F$ have the same recession cone, the constraint is satisfied if and only if all points of a V-representation of $F(\bar x)+C$ belong to $F(x)+C$.
The linear program can be used to compute minimizers for \eqref{eq:p}, see Algorithm \ref{alg:minimizer}.
In this section we assume that $C$ is a polyhedral convex cone with nonempty interior.
Below in Section \ref{sec:empty} we also consider ordering cones which are allowed to have empty interior.

A {\em minimal system of outer normals} of a polyhedral convex set $P \subseteq \R^q$ is the family of vectors 
$\Set{w^1,\dots,w^k}$
arising from an inequality representation of $P$ of the form
$$ P = \Set{y \in \R^q \given y^T w^1 \leq \gamma_1, \dots, y^T w^k \leq \gamma_k },\quad \|w^1\|=\ldots=\|w^k\|= 1,$$
where $k$ is minimal.
If the convex polyhedron $P$ has nonempty interior, the vectors $w^i$ are just the uniquely defined normalized outer normals at the facets of $P$.
The case where $P$ is allowed to have empty interior is considered below in Section \ref{sec:empty}.

\begin{algorithm}[ht]
\DontPrintSemicolon
\SetKwFunction{FMinimizer}{ComputeMinimizer}
\SetKwProg{Fn}{Function}{:}{\KwRet $\Set{\bar x}$}
\SetKwInOut{Input}{input}\SetKwInOut{Output}{output}
\Fn{\FMinimizer{$F$, $C$, $\bar y$}}{
		choose $\bar x \in \R^n$ with $\bar y \in F(\bar x)+C$\;\label{alg1_line_2}
		$N \leftarrow \Set{\text{minimal system of outer normals of $F(\bar x)+C$}}$\;
		$W \leftarrow \emptyset$\;
		\Repeat{$N \subseteq W$}
		{
			choose $w \in N \setminus W$\;
			$W \leftarrow W \cup \Set{w}$\;			
			\eIf{\eqref{eq:lp} has an optimal solution}
			{
				assign it to $(x^*,y^*)$
			}
			{
				\KwRet $\emptyset$
			}
			\If{$y^* \not \in F(\bar x) + C$}
			{
				$\bar x \leftarrow x^*$\; \label{line_update}
				$N \leftarrow \Set{\text{minimal system of outer normals of $F(\bar x)+C$}}$\;				
			}			
		}         
}
\caption{Computation of minimizer $\bar x$ with $\bar y \in F(\bar x)+C$.}
\label{alg:minimizer}
\end{algorithm}

\begin{rem} The execution of Line 2 in Algorithm \ref{alg:minimizer} can be achieved through the resolution of a linear program, while the operations outlined in lines 3 and 15 can be realized by solving the (geometric dual of) a vector linear program. An efficient and practical approach consists of using the polyhedral calculus toolbox {\tt bensolve tools} \cite{bt, CirLoeWei19}.	
\end{rem}

\begin{prop}\label{prop:minimizer}
For a polyhedral convex set-valued mapping $F:\R^n \rightrightarrows \R^q$,
a polyhedral convex ordering cone $C\subseteq \R^q$ with nonempty interior and a point $\bar y$ in the upper image $\PP$ of \eqref{eq:p},
the function {\tt ComputeMinimizer} defined in Algorithm \ref{alg:minimizer},
after finitely many steps, either returns a minimizer $\bar x$ of \eqref{eq:p} with the property $\bar y \in F(\bar x) + C$ or an empty set.
The number of iterations is bounded above by the minimal number of inequalities in an inequality representation of the graph of the mapping $F_C$. 
\end{prop} 

\begin{proof} Assume $\bar x$ is returned but is not a minimizer of \eqref{eq:p}.
	Then there is $\tilde x \in \R^n$ such that $F(\tilde x)+C \supsetneq F(\bar x)+C$. 
	In particular, there is some $\tilde y \in F(\tilde x)+C$ with $\tilde y \not\in F(\bar x)+C$.
	After a regular termination of the loop, the set $N \subseteq W$ is a system of outer normals of $F(\bar x)+C$. 
	Hence there is $\tilde w \in N$ such that
	\begin{equation}\label{eq:dota0}
		\tilde w^T \tilde y > \sup\Set{\tilde w^T z \given z \in F(\bar x)+C} \eqqcolon \tilde \gamma.
	\end{equation}
	Since $\tilde w \in W$, some linear program (LP($F, \tilde w, \bar x(\tilde w)$)) has been solved in some iteration of the loop.
	Its optimal solution (which exists as some $\bar x$ was returned) is denoted by $(x^*(\tilde w),y^*(\tilde w))$. 
	Because of the second constraint in the linear program we can show inductively that
	\begin{equation}\label{eq:dota1}
	F(\bar x(\tilde w))+ C \subseteq F(\bar x)+C
	\end{equation}
	holds. 
	We next show that
	\begin{equation}\label{eq:dota2}
		y^*(\tilde w) \in F(\bar x)+C.
	\end{equation}
	If $y^*(\tilde w) \in F(\bar x(\tilde w))+ C$, this follows from \eqref{eq:dota1}. 
	Otherwise, if $y^*(\tilde w) \not\in F(\bar x(\tilde w))+ C$, the update step in line \ref{line_update} is executed,
	which yields $y^*(\tilde w) \in F(x^*(\tilde w))+ C \subseteq F(\bar x)+C$ by similar arguments as used to show \eqref{eq:dota1}.
	Therefore, \eqref{eq:dota2} holds in any case.
	It implies that the optimal value $\tilde w^T y^*(\tilde w)$ of (LP($F, \tilde w, \bar x(\tilde w)$)) is
	not larger than $\tilde\gamma$ defined in \eqref{eq:dota0}.
	From \eqref{eq:dota1}, we conclude that $(\tilde x,\tilde y)$ is feasible for (LP($F, \tilde w, \bar x(\tilde w)$)).
	Thus \eqref{eq:dota2} yields a contradiction because the value $\tilde w^T \tilde y$ is larger than the optimal value.
	
	Before the loop is entered, we have $\bar y \in F(\bar x) + C$. This property is maintained by the second constraint of the linear program.

	It remains to show that the loop terminates after finitely many steps for each element $\bar y$ of $\PP$. 
	The graph of the set-valued mapping $F_C$ can be expressed by a finite number $m$ of affine inequalities, say 
	$$ \gr F_C = \Set{(x,y) \in \R^n\times \R^q \given A x + B y \leq b}.$$
	The values of $F_C$ can be represented as
	$$ F(\bar x)+ C = \Set{y \in \R^q \given B y \leq b-A \bar x}.$$
	Let $M$ be the set of all normalized column vectors of $B^T$. 
	Since $C$ has nonempty interior, the interior of $F(\bar x)+C$ is nonempty during the algorithm.
	A minimal system $N$ of outer normals of $F(\bar x)+C$ consists of all normalized facet outer normals and thus must be contained in $M$.
	Thus, for every linear program \eqref{eq:lp} solved in the algorithm, the parameter $w$ belongs to $M$.
	The algorithm ensures that at most one linear program is solved for each $w \in M$.
	Therefore the loop terminates after at most $m$ iterations.
\end{proof}

\section{Simplified algorithm for the bounded case} \label{sec:alg_bounded}

In this section we derive a simplified variant of the solution method of \cite{LoeSch13}.
Suppose problem \eqref{eq:p} is bounded. Then the linear program \eqref{eq:lp} has an optimal solution whenever $w \in C^*$ and $\bar x \in \dom F$.
Let the polyhedral convex ordering cone $C\subseteq \R^q$ have nonempty interior. Then Proposition \ref{prop:minimizer} implies that for every $\bar y \in \PP$ the function {\tt ComputeMinimizer} in Algorithm \ref{alg:minimizer} returns a minimizer $\bar x$ with $\bar y \in F(\bar x)+C$. Note that $\bar y \in \PP$ implies $\bar x \in \dom F$ whereas $w \in C^*$ follows from the fact that every parameter $w$ in Algorithm \ref{alg:minimizer} is an outer normal of a nonempty set of the form $F(\bar x) + C$.

Let us assume as in \cite{LoeSch13}, but just for simplicity, that $C$ is free of lines (see Section \ref{sec:alg_unbounded} for the case of more general cones). Then the upper image $\PP$ has a vertex and can be expressed by its vertices $\vertices(\PP)$ as
$$ \PP = \vertices(\PP) + C .$$
The set $\bar S$ computed by Algorithm \ref{alg:bounded} has the property that the set 
$$ C + \conv \bigcup_{\bar x\in \bar S} F(\bar x)$$
contains all vertices of $\PP$. Together this yields the infimum attainment property \eqref{eq:infimizer}. Since all elements of $\bar S$ are minimizers, $\bar S$ is a solution to \eqref{eq:p}. 

These arguments provide the following statement.
 
\begin{prop}
	For a feasible bounded problem \eqref{eq:p} with a polyhedral convex ordering cone $C$ which has nonempty interior and is free of lines, Algorithm \ref{alg:bounded} is correct and terminates after finitely many steps.
\end{prop}

\begin{algorithm}[ht]
\DontPrintSemicolon
\SetKwFunction{FMinimizer}{ComputeMinimizer}
\SetKwInOut{Input}{input}\SetKwInOut{Output}{output}
\Input{$F:\R^n \rightrightarrows \R^q$ (objective mapping), $C \subseteq \R^q$ (ordering cone)}
\Output{solution $\bar S$ of \eqref{eq:p}}
\Begin{
	$\bar S \leftarrow \emptyset$\;
	\For{$y \in \vertices \PP$}{
		$\bar S \leftarrow \bar S \; \cup$ \FMinimizer{$F$, $C$, $y$}\;
	}
}
\caption{Solution method for bounded problems}
\label{alg:bounded}
\end{algorithm}

The computation of the vertices of $\PP$ can be realized by solving a vector linear program, the so-called {\em vectorial relaxation}, which is the problem to $C$-minimize the vector valued function $(x,y) \mapsto y$ under the polyhedral convex constraint $y \in F(x)$. This vector linear program has the same upper image as $\eqref{eq:p}$, see \cite{LoeSch13,LoeSch15}.

\section{Extension to unbounded problems}\label{sec:alg_unbounded}

The aim of this section is to drop the boundedness assumption for \eqref{eq:p} and to derive a solution method for not necessarily bounded polyhedral convex set optimization problems. This problem class has been studied in \cite{HeyLoe,Loe23exsol,Loe23natcon}. It extends the class of not necessarily bounded vector optimization problems, see, e.g., \cite{Loe11}.

In contrast to scalar linear programming, where optimal solutions exist for bounded problems only, in vector linear programming and in polyhedral convex set optimization it makes sense to define (optimal) solutions also for some unbounded problems \cite{Loe11,Loe23exsol,Loe23natcon}. Such a solution of an unbounded problem, if it exists, consists of two components: points as in the bounded case and additionally directions. Most of the new aspects can be described by the {\em recession mapping} of $F$ and by the {\em homogeneous problem} associated to \eqref{eq:p}. 

The recession cone of $\gr F$, denoted by $0^+ \gr F$, is again the graph of a set-valued mapping, which is called the \emph{recession mapping} of $F$.
It is denoted by $G:\R^n\rightrightarrows \R^q$ and defined by the equation
$$ \gr G = 0^+\gr F.$$
If $F$ is replaced by $G$, we obtain the \emph{homogeneous problem}
\leqnomode
\begin{gather}\label{eq:q}\tag{P$\homog$}
  \min G(x)+C \quad \text{ s.t. }\; x \in \R^n.
\end{gather}
\reqnomode
with upper image
$$ \QQ \coloneqq C + \bigcup_{x \in \R^n} G(x).$$
One can show that $\QQ = 0^+ \PP$, see, e.g., \cite{HeyLoe}. Moreover, we have $G(0) = 0^+F(x)$ for all $x \in \dom F$, see, e.g., \cite[Proposition 2]{HeyLoe}.

The following definition of a solution for \eqref{eq:p} is due to \cite{Loe23exsol}\footnote{contributed by Andreas H. Hamel and Frank Heyde}.
A tuple $(\bar{S},\hat{S}$) of finite sets $\bar{S} \subseteq \dom F$,
$\bar{S} \neq \emptyset$ and $\hat{S} \subseteq \dom G \setminus \{0\}$ is called a {\em finite infimizer} \cite{HeyLoe} for problem \eqref{eq:p} if 
\begin{equation*}
	\PP \subseteq C + \conv \bigcup_{x \in \bar{S}} F(x) + \cone \bigcup_{x \in \hat{S}} G(x),
\end{equation*} 
where $\cone Y$ denotes the \emph{conic hull}, i.e.,
the smallest (with respect to $\subseteq$) convex cone containing $Y$.
We set $\cone \emptyset \coloneqq \Set{0}$.
An element $\bar{x} \in \dom F$ is called {\em minimizer} or {\em minimizing point} of \eqref{eq:p} if 
$$ \not\exists x \in \R^n:\; F(x) + C \supsetneq F(\bar x)+C.$$
A nonzero element $\hat{x} \in \dom G$ is called {\em minimizer} or {\em minimizing direction} for \eqref{eq:p} if
$$ \not\exists x \in \R^n:\; G(x)+C \supsetneq G(\hat x)+C.$$
A finite infimizer $(\bar{S},\hat{S})$ is defined to be a {\em solution} to \eqref{eq:p} if all elements of  $\bar{S}$ are minimizing points and all elements of $\hat{S}$ (if any) are minimizing directions.

Conditions which characterize the existence of solutions have been established in \cite{Loe23exsol,Loe23natcon}. The \emph{natural ordering cone} \cite{Loe23natcon} of a polyhedral convex set-valued mapping $F: \R^n\rightrightarrows \R^q$, defined as 
$$ K \coloneqq  K_F \coloneqq \Set{y \in \R^q \given \exists x \in \R^n:\; y \in G(x),\; 0 \in G(x)},$$
provides an elegant way to formulate conditions being equivalent to the existence of a solution to \eqref{eq:p}.

The homogeneous linear program associated to \eqref{eq:lp}, obtained by setting all right hand side entries of the constraints to zero, is
\leqnomode
\begin{gather}\label{eq:lphom}\tag{LP$\homog(F,w)$}
	 \max_{x,y} w^T y \quad \text{ s.t. } \quad  \left\{ \begin{array}{rcl}
		y & \in & G(x)+C \\
		0 & \in & G(x)+C\text{.}
	\end{array}\right.
\end{gather}
\reqnomode
It can be interpreted as maximizing $w^T y$ subject to $y$ belonging to the natural ordering cone of the mapping $F_C$. 

The following two propositions are used to prove Theorem \ref{thm1} below.
The first one is a fact about the existence of optimal solutions of a linear program, which follows, for instance, by linear programming duality.
The second one, taken from \cite{Loe23exsol}, characterizes the existence of solutions of a polyhedral convex set optimization problem by the property that the origin is a minimizer of the homogeneous problem.

\begin{prop}\label{prop:lp}
	A linear program has an optimal solution if and only if it is feasible and its homogeneous problem is bounded.
\end{prop}
%

\begin{prop}[{\cite[Theorem 3.3]{Loe23exsol}}]\label{prop:ex}
	Let $F:\R^n\rightrightarrows \R^q$ be a polyhedral convex set-valued mapping with recession mapping $G$ and let $C \subseteq \R^q$ be a polyhedral convex cone. 
	Then \eqref{eq:p} has a solution if and only if it is feasible and
	$$ \not\exists x \in \R^n:\; G(x) + C \supsetneq G(0)+C.$$
\end{prop}

The following theorem characterizes the existence of solutions of the polyhedral convex set optimization problem by a condition for the natural ordering cone and by the solution behavior of the linear programs which occur in the minimizer computation method in Algorithm \ref{alg:minimizer}. Related results were obtained in \cite{Loe23natcon}. In contrast, we suppose condition \eqref{eq:g0} here, which allows a significant simplification. Below in Section \ref{sec:std} we discuss how the original polyhedral convex set optimization \eqref{eq:p} can be reformulated into a {\em standard form} in which \eqref{eq:g0} is always satisfied. 

\begin{thm}\label{thm1}
	Consider problem \eqref{eq:p} where the set-valued objective mapping $F:\R^n\rightrightarrows \R^q$ is polyhedral convex, $G$ denotes its recession mapping and the ordering cone $C \subseteq \R^q$ is polyhedral convex. Let the objective mapping and the ordering cone be connected by satisfying the condition
	\begin{equation}\label{eq:g0}
		 G(0) = C\text{,}
	\end{equation}
	and let \eqref{eq:p} be feasible. Then, the following statements are equivalent.
	\begin{enumerate}[(i)]
		\item \eqref{eq:p} has a solution.
		\item $K \subseteq C$.
		\item For all $w \in C^*$ and all $\bar x \in \dom F$, \eqref{eq:lp} has an optimal solution.
	\end{enumerate} 
\end{thm}
\begin{proof}
	(i) $\Rightarrow$ (ii). Assume (ii) is violated. Then there is some $y \in K \setminus C$.
	 By the definition of $K$, there exists $x \in \R^n$ such that $y \in G(x) \subseteq G(x)+C$
	  and $0 \in G(x) \subseteq G(x)+C$. 
	From the latter statement and the fact $G(0) + G(x) = G(x)$ we obtain $G(0) + C \subseteq G(0) + G(x) + C + C = G(x) + C$.
	The inclusion is strict, since $y \in G(x)+C$ and, by \eqref{eq:g0}, $y \not\in C = G(0)+C$. 
	By Proposition \ref{prop:ex}, (i) is violated.
	
	(ii) $\Rightarrow$ (iii). Clearly, \eqref{eq:lp} is feasible for every $\bar x \in \dom F \neq \emptyset$. 
	By Proposition \ref{prop:lp}, it remains to show that the homogeneous problem is bounded.
	Let $w \in C^*$ and $\bar x \in \dom F$. By (ii), we have $w \in K^*$.
	By \eqref{eq:g0}, we have $F=F_C$ and thus \eqref{eq:lphom} can be expressed equivalently as
	$$ \max_{y} w^T y \text{ s.t. } y \in K,$$
	which is a bounded problem, by the definition of the polar cone $K^*$.
	 
	(iii) $\Rightarrow$ (i). Assume \eqref{eq:p} has no solution. 
	By Proposition \eqref{prop:ex} there exists some $\bar x \in \R^n$ such that  $G(\bar x)+C \supsetneq G(0)+C \supseteq C$. 
	Thus there is some $\bar y \in G(\bar x)+C$ such that $y \not \in C = (C^*)^*$. 
	This yields the existence of $\bar w \in C^*$ with $\bar w^T \bar y > 0$. 
	The point $(\bar x,\bar y)$ is feasible for \eqref{eq:lphom}. 
	Since the feasible set of a homogeneous problem is a cone, \eqref{eq:lphom} must be unbounded. 
	By Proposition \eqref{eq:lp} this implies that (iii) is violated.	
\end{proof}

The minimizer computation method from Section \ref{sec:minimizer} does not depend on \eqref{eq:p} being bounded. 
By Proposition \ref{prop:minimizer}, the function {\tt ComputeMinimizer} returns a minimizer $\bar x$ for suitable arguments
if and only if all linear programs which occur during the finite procedure have optimal solutions.
Thus, as a consequence of Theorem \ref{thm1}, we obtain the following result.

\begin{cor}\label{cor:minimizer}
	Let $F:\R^n \rightrightarrows \R^q$ be a polyhedral convex set-valued mapping 
	and let $C\subseteq \R^q$ be a polyhedral convex ordering cone with nonempty interior such that \eqref{eq:g0} is satisfied.
	For every $\bar y \in \PP$,
	the function {\tt ComputeMinimizer} defined in Algorithm \ref{alg:minimizer} terminates after finitely many steps. 
	It returns a minimizing point $\bar x$ of \eqref{eq:p} with $\bar y \in F(\bar x) + C$ if \eqref{eq:p} has a solution.
	Otherwise, an empty set is returned.
\end{cor}

Applying this result to the homogeneous problem, we obtain the second part required for the solution method in Algorithm \ref{alg3}.

\begin{cor}\label{cor:minimizerdir}
	Let $F:\R^n \rightrightarrows \R^q$ be a polyhedral convex set-valued mapping 
	and let $C\subseteq \R^q$ be a polyhedral convex ordering cone with nonempty interior such that \eqref{eq:g0} is satisfied.
	For every $\hat y \in \QQ \setminus C$,
	the function {\tt ComputeMinimizer} defined in Algorithm \ref{alg:minimizer} for the arguments $(G,C,\hat y)$ terminates after finitely many steps. 
	It returns a nonzero minimizing direction $\hat x \in \R^n\setminus\Set{0}$ of \eqref{eq:p} with $\hat y \in G(\hat x) + C$ whenever \eqref{eq:p} has a solution.
	If \eqref{eq:p} is feasible but has no solution, it returns an empty set.
\end{cor}
\begin{proof}
	Since the homogeneous problem \eqref{eq:q} of \eqref{eq:p} is just a particular case of a polyhedral convex set optimization problem, Corollary \ref{cor:minimizer} can be applied to \eqref{eq:q}. Condition \eqref{eq:g0} is the same for \eqref{eq:q} of \eqref{eq:p}. Note that (LP$\homog (F,w)$) is the same problem as (LP$\homog (G,w)$) and thus by Proposition \ref{prop:lp} and Theorem \ref{thm1}, in case \eqref{eq:p} is feasible, \eqref{eq:q} has the same solution behavior as \eqref{eq:p}. So, if \eqref{eq:p} is feasible, after finitely many steps, {\tt ComputeMinimizer} returns a minimizing point $\hat x \in \R^n$ of \eqref{eq:q} with $\hat y \in G(\hat x) + C$ whenever a solution of \eqref{eq:p} exists. Otherwise it returns the empty set.
	
	It remains to show that, in case \eqref{eq:p} has a solution, $\hat y \not \in C$ implies $\hat x \neq 0$.
	Indeed, if $\hat x = 0$, we have $0 \in G(0)+C$ and $\hat y \in G(0)+C$. Thus $\hat y \in K$ and, by Theorem \ref{thm1}, $\hat y \in C$.
\end{proof}

The solution method in Algorithm \ref{alg3} and the following verification of it provide the main result of this paper.

\begin{algorithm}[hpt]
\DontPrintSemicolon
\SetKwFunction{FMinimizer}{ComputeMinimizer}
\SetKwInOut{Input}{input}\SetKwInOut{Output}{output}
\Input{$F$ (objective mapping), $C \subseteq \R^q$ (ordering cone)}
\Output{solution $(\bar S,\hat S)$ of \eqref{eq:p} if there is one, $(\bar S,\hat S)=(\emptyset,\emptyset)$ otherwise}
\Begin{	
	$\bar S \leftarrow \emptyset$, $\hat S \leftarrow \emptyset$\;
	\If{$\PP = \emptyset$ or \FMinimizer{$G$,$C$,$0$}$=\emptyset$}
	{
		stop (there is no solution)\; \label{line:4}
	}
	compute a V-representation of $\PP$, store the points in $P$ and the directions not belonging to $C$ in $D$\;
	\For{$y \in P$}{
		$\bar S \leftarrow \bar S \;\cup$ \FMinimizer{$F$,$C$,$y$}\;
	}
	\For{$y \in D$}{
		$\hat S \leftarrow \hat S \;\cup$ \FMinimizer{$G$,$C$,$y$}\;
	}
}
\caption{Solution method for not necessarily bounded problems satisfying \eqref{eq:g0}}
\label{alg3}
\end{algorithm}

\begin{thm}\label{thm2}
	Let $F:\R^n \rightrightarrows \R^q$ be a polyhedral convex set-valued mapping $F:\R^n \rightrightarrows \R^q$
	and let $C\subseteq \R^q$ be a polyhedral convex ordering cone $C\subseteq \R^q$ with nonempty interior such that \eqref{eq:g0} is satisfied. Then Algorithm \ref{alg3} is correct and terminates after finitely many steps.
\end{thm}
\begin{proof}
	Let $(P,D)$ be a V-representation of $\PP$. Then, $\PP$ can be expressed as	
	$$ \PP = \conv P + \cone{D} = \conv P + \cone{(D\setminus C)} + C.$$
	By Corollary \ref{cor:minimizerdir}, the method stops in line \ref{line:4} if and only if \eqref{eq:p} has no solution. Otherwise, if \eqref{eq:p} has a solution, by Corollaries \ref{cor:minimizer} and \ref{cor:minimizerdir}, the function {\tt ComputeMinimizer} always returns for $\bar y \in P$ a minimizing point  $\bar x \in \R^n$ of \eqref{eq:p} with $\bar y \in F(\bar x)+C$ and for $\hat y \in D\setminus C$ a minimizing direction $\hat x \in \R^n \setminus\Set{0}$ of \eqref{eq:p} with $\hat y \in F(\hat x)+C$. This yields that $(\bar S,\hat S)$ is a solution in this case. The Algorithm is finite since $P$ and $D$ are finite sets and the function {\tt ComputeMinimizer} terminates after finitely many steps. 
\end{proof}

Note that the case of bounded problems is included in Algorithm \ref{alg3} and boundedness is even verified or disproven: Problem \eqref{eq:p} is feasible and bounded if and only 
Algorithm \ref{alg3} returns $(\bar S,\hat S)$ with $\bar S \neq \emptyset$ and $\hat S = \emptyset$. In this case $\bar S$ is a solution as defined in Section \ref{sec:alg_bounded}. 

\section{Standard form reformulation}\label{sec:std}

We introduce in the section a {\em standard form} of \eqref{eq:p} which is in some sense equivalent to \eqref{eq:p} and show that condition \eqref{eq:g0} in Theorems \ref{thm1} and \ref{thm2} always holds for such a reformulation of the original problem \eqref{eq:p}.

Let $F:\R^n\rightrightarrows \R^q$ be a polyhedral convex set-valued mapping and let $C\subseteq \R^q$ be a polyhedral convex ordering cone. Define
$$ F\std : \R^n \rightrightarrows \R^q,\quad F\std(x)\coloneqq F(x)+C,\qquad C\std \coloneqq C + G(0)$$
and consider the polyhedral convex set optimization problem
\leqnomode
\begin{gather}\label{eq:pstd}\tag{P$\std$}
  \min F\std(x) + C\std \quad \text{ s.t. }\; x \in \R^n.
\end{gather}
\reqnomode 
\eqref{eq:pstd} is called {\em standard form} of \eqref{eq:p}.

\begin{prop}
	The upper images of \eqref{eq:p} and \eqref{eq:pstd} coincide, that is, $$\PP=\PP\std.$$
\end{prop}
\begin{proof}
	As shown, for instance, in \cite[Proposition 2]{HeyLoe}, $G(0)$ is the recession cone of $F(x)$ for all $x \in \dom F$. Thus,
	\begin{align}
		\PP &= C + \bigcup_{x\in \R^n} F(x) = C + \bigcup_{x\in \R^n} (F(x) + C)  \nonumber \\
		    &= C + \bigcup_{x\in \R^n} (F(x) + G(0) + C) = C +  G(0) +\bigcup_{x\in \R^n} (F(x) + C) = \PP\std, \nonumber
	\end{align}
	which completes the proof.
\end{proof}

\begin{prop} A tuple $(\bar S, \hat S)$ is an infimizer for \eqref{eq:p} if and only if it is an infimizer of \eqref{eq:pstd}.
\end{prop}
\begin{proof}
	By \cite[Propositions 2 and 4]{HeyLoe}, $G(0)$ is the recession cone of $\conv\bigcup_{\bar x \in \bar S} F(x)$.  Thus,
	\begin{align}
		\PP\std = \PP & \subseteq C + \conv \bigcup_{x\in \bar S} F(x) + \cone \bigcup_{x \in \hat S} G(x) \nonumber \\
		            &= C + G(0) + \conv \bigcup_{x\in \bar S} (F(x) + C) + \cone \bigcup_{x \in \hat S} (G(x)+C) \nonumber\\
					&= C\std + \conv \bigcup_{x\in \bar S} F\std(x) + \cone \bigcup_{x \in \hat S} G\std(x), \nonumber
	\end{align}
	which proves the claim.
\end{proof}

\begin{prop} \label{prop:std_inf} 
	A point $\bar x$ (direction $\hat x$) is a minimizer for \eqref{eq:p} if and only if it is a minimizer for \eqref{eq:pstd}. 
\end{prop}
\begin{proof}
	Since
	$$ F(x) + C = F\std(x) + C\std \quad \text{and} \quad G(x) + C = G\std(x) + C\std\text{,}$$
	this follows from the definition of a minimizer.
\end{proof}

\begin{prop} \label{prop:std_min} 
	A tuple $(\bar S, \hat S)$ is a solution to \eqref{eq:p} if and only if it is a solution to \eqref{eq:pstd}.
\end{prop}
\begin{proof}
	Follows from Propositions \ref{prop:std_inf} and \ref{prop:std_min}
\end{proof}

\begin{prop} For the standard form \eqref{eq:pstd} of \eqref{eq:p}, one has $G\std(0) = C\std$.		
\end{prop}
\begin{proof}
	This is obvious as $G\std (0) = G(0)+C = C\std$.
\end{proof}

We conclude that for a problem given in standard form, condition \eqref{eq:g0} can be omitted in all of the above statements.

\section{Ordering cones with empty interior}\label{sec:empty}

In this last section we drop the assumption that the ordering cone $C$ is required to have nonempty interior.
This assumption was used in Section \ref{sec:minimizer} to show finiteness of the minimizer computation.
All subsequent results use this assumption only in order to use Algorithm \ref{alg:minimizer} and its verification in Proposition \ref{prop:minimizer}. 

Let $C$ be an arbitrary polyhedral convex cone.
Then, the set $F(\bar x)+C$ used in Algorithm \ref{alg:minimizer} can have empty interior.
This means that the vectors in a minimal system of outer normals as defined in Section \ref{sec:minimizer} are not necessarily uniquely defined. 

Any polyhedral convex set $P \subseteq \R^q$ can be expressed as
$$ P = (P + (\aff P)^\bot) \cap \aff P,$$
where $\aff P$ denotes the affine hull of $P$ and $(\aff P)^\bot$ its orthogonal complement. 
The polyhedron $(P + (\aff P)^\bot)$ has full dimension. 
Hence it has a uniquely defined  minimal system of outer normals $\Set{w^1,\dots,w^m}$ (which can be empty).
For this part of an inequality representation of $P$, Algorithm \ref{alg:minimizer} works analogously.

Let $\dim P = r < q$. Then $\aff P$ can always be expressed by $q-r \geq 1$ affine equations and this number is minimal. If $\aff P = \Set{y \in \R^q \given A y = a}$ is a minimal representation by equations, then $\aff P = \Set{y \in \R^q \given A y \leq a, - e^T A y \leq -e^T a}$, where $e^T=(1,\dots,1)$, is a minimal inequality representation.

Since a minimal inequality representation of $\aff P$ is not necessarily uniquely defined,
we need another argumentation for the finiteness in Algorithm \ref{alg:minimizer}.
Algorithm \ref{alg:minimizer_repl} shows how Algorithm \ref{alg:minimizer} can be modified in order to become able to treat ordering cones with empty interior.
For simplicity we assume in Algorithm \ref{alg:minimizer_repl} that $\eqref{eq:p}$ has a solution, i.e., all linear programs have optimal solutions. If this is not the case, one just needs to return the empty set as in the remaining part of Algorithm \ref{alg:minimizer}.

Even though the function $x \mapsto \dim (F(x)+C)$ is constant on the relative interior of $\dom F$ and not larger elsewhere, see, e.g., \cite[Proposition 3.7]{Loe06}, it is not constant on the whole domain of $F$. 
The goal of Algorithm \ref{alg:minimizer_repl} is to determine for given $\bar y \in \PP$ some $\bar x \in \R^n$ with $\bar y \in F(\bar x)+C$ such that
\begin{equation}\label{eq:dim}
	\not\exists x \in \R^n:\,F(x) + C \supseteq F(\bar x)+C \, \text{ and } \, \dim(F(x)+C)>\dim(F(\bar x) + C).
\end{equation} 
This implies that the affine hull of $F(\bar x)+C$ is constant in all remaining steps of Algorithm \ref{alg:minimizer}.

\begin{algorithm}[ht]
\DontPrintSemicolon
\SetKwInOut{Input}{input}\SetKwInOut{Output}{output}
\Input{$F:\R^n \rightrightarrows \R^q$ (objective mapping), $C \subseteq \R^q$ (ordering cone), $\bar y \in \PP$}
\Output{modified $F:\R^n \rightrightarrows \R^q$, some $\bar x \in \R^n$ with $\bar y \in F(\bar x)+C$}
	choose $\bar x \in \R^n$ with $\bar y \in F(\bar x)+C$\;
	\Repeat{leave}
	{
		$N \leftarrow \Set{\text{minimal system of outer normals of $\aff(F(\bar x)+C)$}}$\;
		leave $\leftarrow$ true\;
		\For{$w \in N$}
		{	
		    $(x^*,y^*) \leftarrow$  optimal solution of \eqref{eq:lp}\;
			\If{$y^* \not \in \aff(F(\bar x) + C)$}
			{
				$\bar x \leftarrow x^*$\; \label{alg4:update}
				leave $\leftarrow$ false\;
				break\;				
			}			
		} 
	}
	replace $F$ by the mapping $x \mapsto F(x)+(\aff (F(\bar x)+C))^\bot$\;        

\caption{Replacement of line \ref{alg1_line_2} in Algorithm \ref{alg:minimizer}.}
\label{alg:minimizer_repl}
\end{algorithm}

\begin{prop}\label{prop:aff}
	Let \eqref{eq:p} have a solution. Then for given $\bar y \in \PP$, Algorithm \ref{alg:minimizer_repl} is finite and it computes some $\bar x\in \R^n$ with $\bar y \in F(\bar x)+C$ such that \eqref{eq:dim} holds for the unmodified mapping $F$.
\end{prop}
\begin{proof}
	Let $\bar x^+$ denote the variable $\bar x$ after an update step in line \ref{alg4:update}. Then we have
	$F(\bar x^+)+C \supseteq F(\bar x)+C$ and $\dim (F(\bar x^+)+C) > \dim (F(\bar x) + C)$. Hence there can be only finitely many such update steps and thus the algorithm is finite.
	By similar arguments as used in the proof of Proposition \ref{prop:minimizer}, at termination of the outer loop we have $\bar y \in F(\bar x)+C$.
	
	Assume \eqref{eq:dim} is not satisfied at termination of the outer loop. Then there is some $\tilde x \in \R^n$ with $F(\tilde x) + C \supseteq F(\bar x)+C$ and $\dim(F(\tilde x)+C)>\dim(F(\bar x) + C)$. There is some $\tilde y \in F(\tilde x)+C$ with $\tilde y \not\in \aff(F(\bar x)+C)$. The point $(\tilde x,\tilde y)$ is feasible for \eqref{eq:lp}.
	Since $N$ is a minimal system of outer normals of $\aff(F(\bar x)+C)$ and $\tilde y \not\in \aff(F(\bar x)+C)$ there exists $\tilde w \in N$ with 
	$$\tilde w^T \tilde y > \sup_{x \in \R^n} \Set{ \tilde w^T y \in \R^q \given y \in F(\bar x)+C} \geq \tilde w^T y^*.$$
	This is a contradiction because $\tilde w^T \tilde y$ is larger than the optimal value $\tilde w^T y^*$ of the linear program (LP($F, \tilde w, \bar x$)).
\end{proof}

\section{Numerical examples} \label{sec:num}

In this section, we discuss two numerical examples. The first one is a specific instance of the motivating example in Section \ref{sec:mot}. The data we have chosen lead to an unbounded polyhedral convex set optimization problem. We kept this example small for illustrative purposes. The second example is computationally more challenging. It is a bounded problem, but we use an ordering cone with an empty interior, which requires applying the results from Section \ref{sec:empty}.

The algorithms of this paper were implemented in GNU Octave. We used the polyhedral calculus package {\tt bensolve tools} \cite{bt, CirLoeWei19}. The experiments were run on a desktop computer with a $3.6$ GHz clock speed.

\begin{ex}\label{ex:001}
	\begin{figure}
		\includegraphics[width=.49\textwidth]{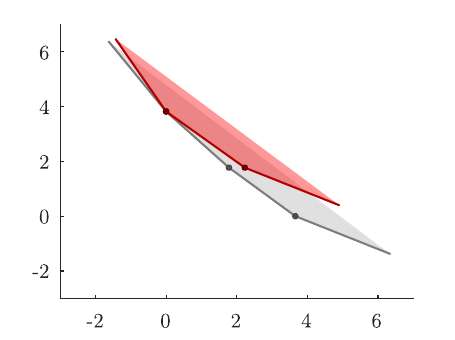}
		\includegraphics[width=.49\textwidth]{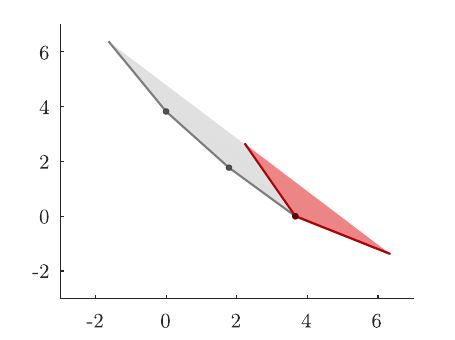}
		\includegraphics[width=.49\textwidth]{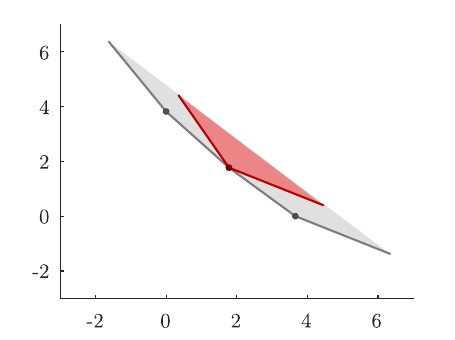}
		\includegraphics[width=.49\textwidth]{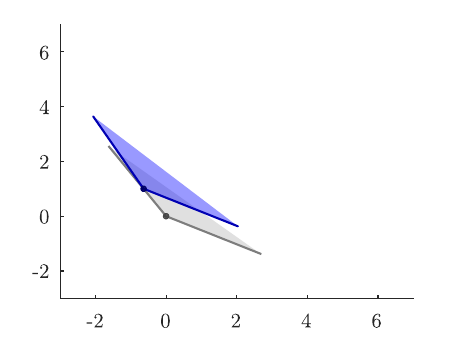}
		\caption{Illustration of Example \ref{ex:001}. The values (red) of the set-valued map $F$ for the three points in $\bar S$ and the value (blue) of the recession map $G$ for the direction in $\hat S$. In the rear the upper image $\PP$ of \eqref{eq:p}, respectively, the upper image $\QQ$ of $\eqref{eq:q}$ are displayed.}\label{fig:001}
	\end{figure}
	In Section \ref{sec:mot} we introduced a set optimization problem which fits into the framework of this paper if the solvency cones $K_1 \subseteq \R^n$ and $K_2 \subseteq \R^n$ are polyhedral convex cones, and if we define a polyhedral convex set-valued objective mapping $F:\R^n \rightrightarrows \R^q$ by 
	$$ F(x) = \left\{y \in \R^q \bigg|\; x \in \Set{\bar x}-K_1,\; x + \begin{pmatrix} y \\ 0 \end{pmatrix} \in K_2\right\}.$$
	A solvency cone $K \subseteq \R^n$ can be constructed from a \emph{bid-ask-matrix} $\Pi \in \R^{n \times n}$ \cite{Schachermayer04}. An entry $\pi_{ij}$ of $\Pi$ denotes the number of units of asset $i$ for which an agent can buy one unit of asset $j$. The following conditions are usually supposed to the bid-ask-matrix $\Pi$ \cite{Schachermayer04}:
$\pi_{ii}=1$, $0<\pi_{ij}$, and $\pi_{ij} \leq \pi_{ik}\pi_{kj}$ for all $i,j,k \in \Set{1,\dots,n}$. 
The solvency cone $K$ induced by the bid-ask-matrix $\Pi$ is spanned by the vectors $\pi_{ij} e^i - e^j$, $i,j \in \Set{1,\dots,n}$, where $e^i$ denotes the $i$-th unit vector in $\R^n$.

Now, let $n=4$ and $q=2$ and the let the solvency cones $K_1$ and $K_2$ be induced, respectively, by the bid-ask-matrices
{ \renewcommand{\arraystretch}{1.5}
\[
	\Pi_1 = \begin{pmatrix}
	1 & \frac{223}{100} & \frac{89}{50} & \frac{47}{25} \\
	\frac{157}{100} & 1 & \frac{54}{25} & \frac{177}{100} \\
	\frac{19}{10} & \frac{211}{100} & 1 & \frac{199}{100} \\
	\frac{101}{50} & \frac{41}{20} & \frac{43}{20} & 1
	\end{pmatrix} \qquad
\Pi_2 = \begin{pmatrix}
1 & \frac{39}{20} & \frac{223}{100} & \frac{207}{100} \\
\frac{37}{20} & 1 & \frac{41}{20} & \frac{54}{25} \\
\frac{38}{25} & \frac{97}{50} & 1 & \frac{37}{20} \\
\frac{11}{5} & \frac{49}{25} & \frac{44}{25} & 1
\end{pmatrix}.
\]
}
Both $K_1$ and $K_2$ each have $12$ extremal directions and $20$ facets. For illustrative purposes, we used relatively large entries in $\Pi_1$ and $\Pi_2$. The solution computed by our implementation is
\[
\bar{S} = \left\{
\begin{pmatrix}
0 \\
-1.7700 \\
-1.0000 \\
0
\end{pmatrix},
\begin{pmatrix}
-3.6600 \\
0 \\
0 \\
0
\end{pmatrix},
\begin{pmatrix}
-1.7800 \\
-1.7700 \\
0 \\
0
\end{pmatrix}
\right\}
\qquad 
\hat{S} = \left\{
\begin{pmatrix}
0.6369 \\
-1.0000 \\
0 \\
0
\end{pmatrix}
\right\}
\]
The presence of a direction in $\hat S$ indicates that the problem instance is unbounded.
The corresponding objective values are displayed in Figure \ref{fig:001}. The runtime was about 0.7 seconds.
\end{ex}

\begin{figure}
	\includegraphics[width=.49\textwidth]{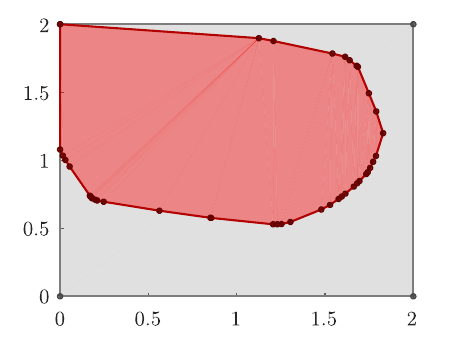}
	\includegraphics[width=.49\textwidth]{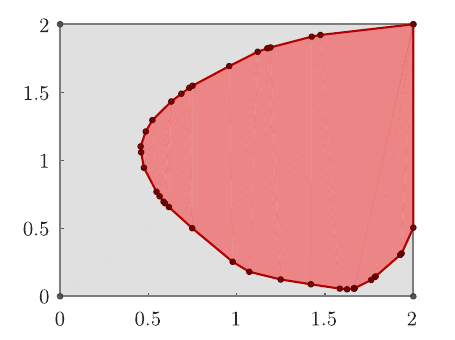}
	\includegraphics[width=.49\textwidth]{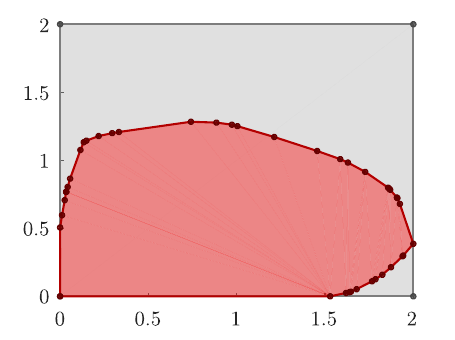}
	\includegraphics[width=.49\textwidth]{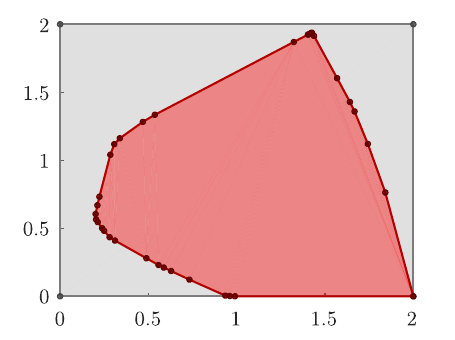}
	\caption{Computational results of Example \ref{ex:002}. The objective values of $F$ for a solution $x^1$, $\dots$, $x^4$ are shown. The square in the rear is the upper image.}\label{fig:002}
\end{figure}

\begin{ex}\label{ex:002}
The graph of $F:\R^{10} \rightrightarrows \R^2$ is taken to be the convex hull of the columns of a $12 \times 60$ matrix $M$ with entries in $\Set{0,1,2}$. To make the example reproducible, $M$ is filled row-wise with the digits of Euler's number $e=2.1718\dots$: $M_{1,1}=2$, $M_{1,2}=7$, $M_{1,3}=1$, $M_{1,4}=8$, $\dots$, $M_{2,1} = 7$, $M_{2,2}=6 $\dots$, M_{12,60} = 5$. Then each entry $M_{ij}$ is replaced by the remainder of the division of $M_{ij}$ by $3$, i.e, $M_{1,1}=2$, $M_{1,2}=1$, $M_{1,3}=1$, $M_{1,4}=2$, $\dots$, $M_{2,1} = 1$, $M_{2,2}=0$, $\dots$, $M_{12,60}=2$. The ordering cone is set to $C = \Set{0} \in \R^2$. The upper image is the square with vertices $v^1=(0,2)$, $v^2=(2,2)$, $v^3=(0,0)$, $v^4=(2,0)$. For every vertex $v^i$, the algorithm computes some  $x^i \in \R^{10}$ such that $v^i \in F(x^i)$, and such that there is no $x \in \R^{10}$ with $F(x) \supsetneq F(x^i)$, $i \in \Set{1,2,3,4}$. This resulting objective values of $F$ are depicted in Figure \ref{fig:002}. The runtime was about 89 seconds. Table \ref{tab:1} provides some information about the number of iterations. 

\begin{table}[h]
\centering
\begin{tabular}{lcccc}
\toprule
\multicolumn{1}{l}{Vertices of the upper image} & $v_1$ & $v_2$ & $v_3$ & $v_4$ \\
\midrule
\multicolumn{1}{l}{Inner iterations of Algorithm \ref{alg:minimizer_repl}} & $1$ & $1$ & $1$ & $2$ \\
\midrule
\multicolumn{1}{l}{Iterations in Algorithm \ref{alg:minimizer}} & $44$ & $38$ & $43$ & $36$ \\
\bottomrule
\end{tabular}
\caption{Number of iterations in Example \ref{ex:002}}\label{tab:1}
\end{table}
\end{ex}

\section{Conclusions}

We presented a solution method for a polyhedral convex set optimization problem in its most general form. Correctness and finiteness are proven without any further assumption than polyhedral convexity of the objective mapping $F:\R^n \rightrightarrows \R^q$ and the ordering cone $C \subseteq \R^q$.


\end{document}